\documentclass{amsart}

\usepackage[T1]{fontenc}
\usepackage[utf8]{inputenc}
\usepackage{lmodern}
\usepackage[provide=*,english]{babel}
\usepackage{amsmath,amssymb,amsthm,mathtools}
\usepackage{microtype}
\usepackage{float}
\microtypesetup{expansion=false}
\usepackage[colorlinks=true,linkcolor=blue,citecolor=blue,urlcolor=blue]{hyperref}

\makeatletter
\def\@@and{and}

\@namedef{subjclassname@2020}{Mathematics Subject Classification
\textup{2020}}

\makeatother

\numberwithin{equation}{section}

\newtheorem{theorem}{Theorem}[section]

\newtheorem{lemma}[theorem]{Lemma}
\newtheorem{maintheorem}{Theorem}

\theoremstyle{definition}
\newtheorem{definition}[theorem]{Definition}
\newtheorem{example}[theorem]{Example}
\theoremstyle{remark}
\newtheorem{remark}[theorem]{Remark}

\DeclareMathOperator{\diam}{diam}
\DeclareMathOperator{\inj}{inj}
\DeclareMathOperator{\dis}{dis}

\DeclareMathOperator{\Lip}{Lip}
\usepackage{tikz}
\usetikzlibrary{arrows.meta,calc}

\begin{document}

\title[Topological entropy and factor persistence]
{Flexibility of topological entropy and factor persistence
in the $C^0$-Gromov Hausdorff topology}

\author{Richard Cubas}
\address{Universidad Científica del Sur, Lima, Peru}
\email{rcubas@cientifica.edu.pe}

\author{Jorge Cris\'ostomo}
\address{Universidad Nacional Mayor de San Marcos, Lima, Peru}
\email{jcrisostomop@unmsm.edu.pe}


\begin{abstract}
We prove that every surjective topological dynamical system on a
compact metric space can be approximated in the $C^0$-Gromov--Hausdorff
topology by homeomorphisms with any prescribed topological entropy
in $[0,+\infty]$, even when simultaneous approximation of its sets
of invariant and ergodic measures is required. We also establish
a rigidity result under uniform geometric bounds in a fixed
dimension: for a homeomorphism of a closed, connected Riemannian
manifold, topological stability in the sense of Gromov--Hausdorff
guarantees its persistence as a factor of sufficiently close
perturbations within the same geometric class. This persistence
prevents topological entropy from decreasing and ensures that
every invariant or ergodic measure of the original system admits
a lift of the same type to the perturbed system.
\end{abstract}

\maketitle
\section{Introduction}

The Gromov--Hausdorff distance provides a natural framework for comparing compact metric spaces without requiring them to be realized as subspaces of a common ambient space; see \cite{burago2001,gromov2007}. To extend this perspective to dynamical systems defined on different metric spaces, Arbieto and Morales \cite{arbieto2017} introduced the $C^0$ Gromov--Hausdorff distance, which we denote by $d_{GH^0}$. This distance allows one to compare both the underlying metric spaces and the dynamical systems acting on them through almost isometric maps that approximately intertwine the corresponding dynamics. Building on this approach, several notions of distance between dynamical systems have been developed, including equivariant and measurable formulations. Related questions concerning variational principles, topological stability relative to probability measures, and the persistence of dynamical properties have also been investigated; see \cite{chung2020,khanMoralesDas2024,leeMoralesNguyen2024,moralesTran2026}.

In this context, a natural question concerns the behavior of topological entropy under perturbations in the $GH^0$ topology. Cubas \cite{cubas2018} showed that topologically transitive homeomorphisms can be approximated in this topology by systems with zero topological entropy. Subsequently, Jung \cite{jung2019} proved that systems consisting of a single periodic orbit are dense, with respect to the $GH^0$ topology, in the class of chain-transitive maps. These results motivate a more general problem: determining which entropy values can be realized by arbitrarily small perturbations of a given system. It is also natural to ask whether such approximations can be achieved while simultaneously controlling the sets of invariant and ergodic measures.

The first main result of this paper provides a precise characterization of continuous maps admitting approximations with arbitrarily prescribed topological entropy. We prove that a continuous map on a compact metric space can be approximated in the $GH^0$ topology by homeomorphisms with any prescribed topological entropy $t\in[0,+\infty]$ if and only if it is surjective. Moreover, for surjective maps, such approximations can be constructed while simultaneously controlling the sets of invariant and ergodic measures.

The proof is based on the construction of a countable compact model with zero topological entropy that approximates both the original dynamics and its sets of invariant and ergodic measures. The desired approximation is then obtained by taking the product of this model with an auxiliary system of arbitrarily small diameter and prescribed entropy. Measures defined on different spaces are compared by pushing them forward to the original space through continuous almost isometric maps.

The second main result  establishes a rigidity phenomenon that contrasts with the flexibility described above. We consider perturbations within a class of closed, connected Riemannian manifolds of fixed dimension, with uniformly bounded sectional curvature and diameter, and injectivity radius uniformly bounded away from zero. We show that if the original homeomorphism is topologically $GH^0$-stable, then it persists as a topological factor of every sufficiently close system within this class. Consequently, the topological entropy of the perturbed system is at least that of the original system. Furthermore, every invariant or ergodic measure of the original system admits an invariant or ergodic lift, respectively, to the perturbed system.

The main difficulty lies in establishing the surjectivity of the semiconjugacy provided by $GH^0$-stability. The geometric assumptions yield a uniform positive lower bound on volume and, after a suitable rescaling of the metrics, allow us to apply the homotopy stability theorem of Mitsuishi and Yamaguchi \cite[Theorem~1.2]{mitsuishiYamaguchi2019}. This result allows us to approximate the semiconjugacy by a Lipschitz homotopy equivalence. Uniform closeness ensures that the two maps are homotopic, and an argument based on top-dimensional homology then establishes surjectivity.

Taken together, these results show that the behavior of topological entropy in the $GH^0$ topology depends essentially on the restrictions imposed on the phase spaces. While perturbations on arbitrary compact metric spaces allow any prescribed entropy value to be realized near a surjective map, uniform geometric constraints, combined with $GH^0$-stability, guarantee the persistence of topological factors and prevent entropy from decreasing. Thus, entropy flexibility and geometric rigidity emerge as two complementary aspects of dynamics in the $C^0$ Gromov--Hausdorff topology.

The paper is organized as follows. Section~\ref{sec:preliminares} recalls the basic definitions and states our main results. In Section~\ref{sec:realizacion}, we construct the countable model and prove Theorem~\ref{thm:A}. Section~\ref{sec:rigidez} contains the measure-lifting lemma and the proof of Theorem~\ref{thm:B}.

\section{Basic Definitions and Results} \label{sec:preliminares}
All metric spaces are assumed to be compact and nonempty.
\subsection{Entropy flexibility}
 If $i:X\to Y$ is a map between metric spaces,
its distortion is
$$
\dis(i)=\sup_{x,x'\in X}
\left|d_Y(i(x),i(x'))-d_X(x,x')\right|.
$$
We do not require $i$ to be continuous. We say that $i$ is an
$\varepsilon$-isometry if
$$
\dis(i)<\varepsilon
\qquad\text{and}\qquad
\sup_{y\in Y}d_Y(y,i(X))<\varepsilon.
$$

Let $f:X\to X$ and $g:Y\to Y$ be continuous maps.
We use the symmetric formulation
\begin{align*}
d_{GH^0}(f,g)=\inf\bigl\{\varepsilon>0&:
 \ \text{there are $\varepsilon$-isometries }i:X\to Y,
j:Y\to X\text{ such that}\\
&\sup_{x\in X}d_Y(g(i(x)),i(f(x)))<\varepsilon \text{ and } \; \sup_{y\in Y}d_X(j(g(y)),f(j(y)))<\varepsilon
\bigr\}.
\end{align*}
See \cite{arbieto2017,chung2020,moralesTran2026}
for this distance and related formulations.
Despite its traditional name, $d_{GH^0}$ is not used here as an
ordinary metric: in general, a triangle inequality
with coefficient $2$ holds; see \cite[Theorem~1]{arbieto2017}.
Our arguments require only the neighborhoods defined by the
inequality $d_{GH^0}(f,g)<\varepsilon$.

\begin{definition}\label{def:estabilidad_GH}
A homeomorphism $f:X\to X$ is \emph{topologically
$GH$-stable} if for every $\varepsilon>0$ there exists $\eta>0$ such that, for every
homeomorphism $g:Y\to Y$ with $d_{GH^0}(f,g)<\eta$, there is a continuous $\varepsilon$-isometry $h:Y\to X$ satisfying
$$
f\circ h=h\circ g.
$$
\end{definition}

We denote by $\mathbb{P}(X)$ the space of probability measures on
$X$. The bounded-Lipschitz distance is
$$
d_{\rm BL}(\mu,\nu)=
\sup\left\{
\left|\int_X\varphi\,d\mu-\int_X\varphi\,d\nu\right|:
\|\varphi\|_\infty\leq1,\ \Lip(\varphi)\leq1
\right\}.
$$
This distance metrizes the weak topology on $\mathbb{P}(X)$ and satisfies
$d_{\rm BL}\leq2$. For nonempty subsets $\mathcal A,\mathcal B$ of
$\mathbb{P}(X)$, we write the Hausdorff distance induced by the
bounded-Lipschitz metric as follows:
$$
d_H^{\rm BL}(\mathcal A,\mathcal B)=
\max\left\{
\sup_{\mu\in\mathcal A}\inf_{\nu\in\mathcal B}d_{\rm BL}(\mu,\nu),
\sup_{\nu\in\mathcal B}\inf_{\mu\in\mathcal A}d_{\rm BL}(\mu,\nu)
\right\}.
$$
When either set is not closed, this expression depends only on
the closures; we shall use it in this sense.

For a continuous map $f:X\to X$, we define
$$
\mathcal M_f(X)=\{\mu\in\mathbb{P}(X):f_*\mu=\mu\}
$$
and denote by $\mathbb{E}_f(X)$ the set of its ergodic elements.
The set $\mathcal M_f(X)$ is compact and convex, and its extreme points
are precisely the elements of $\mathbb{E}_f(X)$. For $x\in X$ and
$m\geq1$, the empirical measure of length $m$ is
$$
\mathsf e_m^f(x)=\frac1m\sum_{k=0}^{m-1}\delta_{f^k(x)}.
$$
We shall repeatedly use the inequality
$$
d_{\rm BL}\left(\sum_i t_i\mu_i,\sum_i t_i\nu_i\right)
\leq \sum_i t_i\,d_{\rm BL}(\mu_i,\nu_i),
$$
valid for $t_i\geq0$ with $\sum_i t_i=1$. It follows
directly from the definition of $d_{\rm BL}$, since for every
admissible test function $\varphi$,
$$
\left|
\int\varphi\,d\left(\sum_i t_i\mu_i-\sum_i t_i\nu_i\right)
\right|
\leq
\sum_i t_i
\left|\int\varphi\,d(\mu_i-\nu_i)\right|
\leq
\sum_i t_i\,d_{\rm BL}(\mu_i,\nu_i).
$$

If $q:Y\to X$ is continuous, we define

$$
\Delta_{\rm BL}^{q}(g,f)=
\max\left\{
d_H^{\rm BL}\bigl(q_*\mathcal M_g(Y),\mathcal M_f(X)\bigr),
d_H^{\rm BL}\bigl(q_*\mathbb{E}_g(Y),\mathbb{E}_f(X)\bigr)
\right\}.
$$

Our first result characterizes the surjectivity of a continuous map through \(C^0\)-Gromov--Hausdorff approximations by homeomorphisms with arbitrarily prescribed topological entropy, while simultaneously controlling the sets of invariant and ergodic measures. More precisely:

\begin{maintheorem}
\label{thm:A}
Let $f:X\to X$ be a continuous map of a compact metric
space. The following are equivalent:
\begin{enumerate}
\item $f$ is surjective.
\item For every $\varepsilon>0$ and every $t\in[0,+\infty]$, there exist a
compact metric space $Y$, a homeomorphism $g:Y\to Y$, and a
continuous $\varepsilon$-isometry $q:Y\to X$ such that
$$
d_{GH^0}(f,g)<\varepsilon,\qquad
h_{\rm top}(g)=t,\qquad
\Delta_{\rm BL}^{q}(g,f)<\varepsilon.
$$
\end{enumerate}
If $f$ is surjective, $Y$ can be chosen to be countable when $t=0$.
\end{maintheorem}

In Theorem~\ref{thm:A}, the control of $\Delta_{\rm BL}^{q}$ expresses
closeness of the sets of measures in the weak topology. It does not imply
that their measure-theoretic entropies are preserved or that the measure-preserving systems are
isomorphic.

The flexibility described by Theorem~\ref{thm:A} contrasts with what
occurs under Walters' topological stability on a fixed space.
If a homeomorphism $f:X\to X$ is topologically stable, then for every
homeomorphism $g:X\to X$ sufficiently close in $C^0$ there is a
continuous map $h:X\to X$, close to the identity, such that
$f\circ h=h\circ g$.
When $X$ is a closed, connected Riemannian manifold, the closeness
of $h$ to the identity forces $h$ to be surjective. Therefore, $(X,f)$ is a
factor of $(X,g)$ and
$$h_{\rm top}(g)\geq h_{\rm top}(f).$$

The topological $GH$-stability of $f:X\to X$
provides, for every sufficiently close homeomorphism
$g:Y\to Y$, a continuous approximately isometric map $h:Y\to X$
satisfying $f\circ h=h\circ g$. However, its image may be
a proper subset of $X$, even though it is dense at the scale of the
approximation. In this case, $h$ does not define a surjective
semiconjugacy from $(Y,g)$ onto $(X,f)$: it only shows that
$(h(Y),f|_{h(Y)})$ is a factor of $(Y,g)$. Thus, one obtains only
$h_{\rm top}(g)\geq h_{\rm top}(f|_{h(Y)})$, without a general comparison
with $h_{\rm top}(f)$.
\subsection{Entropy rigidity}
For $n\geq1$ and $K,i_0,D>0$, we denote by
$$
\mathcal R(n,K,i_0,D)
=
\left\{
\begin{array}{l|l}
M &
\begin{array}{l}
M\text{ is a closed, connected Riemannian manifold},\\
\dim M=n,\quad |\sec_M|\leq K,\\
\inj(M)\geq i_0,\quad \diam(M)\leq D
\end{array}
\end{array}
\right\}.
$$
Here, closed means compact and without boundary. The condition
$|\sec_M|\leq K$ means that the sectional curvature of every
two-dimensional plane in each tangent space of $M$ belongs
to $[-K,K]$. Moreover,
$$
\inj(M)=\inf_{p\in M}\inj_M(p)
$$
is the injectivity radius of $M$, where $\inj_M(p)$ is the
supremum of the radii $r>0$ for which the exponential
map $\exp_p$ is a diffeomorphism from the ball
$B_{T_pM}(0,r)$ onto its image. Finally,
$$
\diam(M)=\sup_{p,q\in M}d_M(p,q),
$$
where $d_M$ is the Riemannian distance.
In dimension one, the sectional curvature condition
is vacuous, since the tangent spaces contain no two-dimensional
planes. 

Our second result establishes that, within the class of Riemannian manifolds $\mathcal R(n,K,i_0,D)$, topological
$GH$-stability yields surjective semiconjugacies under
sufficiently small perturbations, together with an entropy
inequality and the lifting of invariant and ergodic measures. More precisely:

\begin{maintheorem}
\label{thm:B}
Let $n\geq1$, $K,i_0,D>0$, and let $f:X\to X$ be a
$GH$-stable homeomorphism, with $X\in\mathcal R(n,K,i_0,D)$. Then there exists
$\delta>0$ such that every homeomorphism $g:Y\to Y$, with
$Y\in\mathcal R(n,K,i_0,D)$ and $d_{GH^0}(f,g)<\delta$, admits
a continuous surjective map $h:Y\to X$ satisfying
$$
f\circ h=h\circ g.
$$
Consequently,
$$
h_{\rm top}(g)\geq h_{\rm top}(f),
\quad
h_*\mathcal M_g(Y)=\mathcal M_f(X)
 \; \text{ and }\;
h_*\mathbb{E}_g(Y)=\mathbb{E}_f(X).
$$
\end{maintheorem}

The equalities concerning measures mean that every invariant measure
of $f$ admits an invariant lift for $g$, and that every
ergodic measure of $f$ admits an ergodic lift. These lifts
are not asserted to be unique.

The two theorems thus show the role of geometric
restrictions. When the space is allowed to vary freely,
the surjectivity of $f$ characterizes the possibility of realizing
any entropy value in $[0,+\infty]$ in every
$GH^0$ neighborhood, with simultaneous approximation of the sets of measures.
By contrast, within a fixed class $\mathcal R(n,K,i_0,D)$,
a $GH$-stable system persists as a factor of every sufficiently
close perturbation, and its invariant and ergodic measures admit
lifts under the factor map. In particular, it is a
local minimum of topological entropy; hence, entropy
is lower semicontinuous at this reference system when
restricted to perturbations within that class. Example~\ref{ex:gato}, based on a hyperbolic toral automorphism,
illustrates both conclusions for the same reference system
and shows that the maps provided by
$GH$-stability may fail to be surjective when perturbations
without geometric restrictions are allowed.

\section{Proof of Theorem A}
\label{sec:realizacion}

To prove Theorem~\ref{thm:A}, we begin by constructing an
auxiliary system with a fixed point whose entropy can take any
value in $[0,+\infty]$. We then choose recurrent points that are
generic for ergodic measures and use them to construct a
countable zero-entropy model approximating the dynamics and its
sets of invariant and ergodic measures. Taking the product
with the auxiliary system yields the desired entropy value.

\begin{lemma}
\label{lem:sistema_auxiliar_entropia}
For every $a\in[0,+\infty]$, there exist a compact metric space
$W_a$, a homeomorphism $T_a:W_a\to W_a$, and a fixed point
$*\in W_a$ such that
$$
h_{\rm top}(T_a)=a.
$$
\end{lemma}

\begin{proof}
For $a=0$, take a singleton $W_0=\{*\}$ and $T_0(*)=*$.
Suppose that $a>0$. For $N\geq2$, let $\Sigma_N=\{1,\ldots,N\}^{\mathbb Z}$,
with the product topology, and let
$\sigma_N:\Sigma_N\to\Sigma_N$ be the two-sided shift.
For $k\geq1$, consider
$$
W(N,k)=\Sigma_N\times\{0,1,\ldots,k-1\}
$$
and define $S_{N,k}:W(N,k)\to W(N,k)$ by
$$
S_{N,k}(\omega,r)=
\begin{cases}
(\omega,r+1),&0\leq r<k-1,\\
(\sigma_N\omega,0),&r=k-1.
\end{cases}
$$
The map $S_{N,k}$ is a homeomorphism. Its iterate
$S_{N,k}^k$ preserves each level $\Sigma_N\times\{r\}$, and
its restriction to that level is conjugate to $\sigma_N$.
Since these levels form a finite cover by closed sets
invariant under $S_{N,k}^k$, we have
$$
h_{\rm top}(S_{N,k}^k)
=h_{\rm top}(\sigma_N)=\log N.
$$
Therefore, $h_{\rm top}(S_{N,k})=\frac{\log N}{k}$. If $0<a<+\infty$, choose integers $k_j\to+\infty$
sufficiently large that
$N_j=\lfloor e^{ak_j}\rfloor\geq2$. Then
$$
\frac{\log N_j}{k_j}\leq a,
\qquad
\frac{\log N_j}{k_j}\longrightarrow a.
$$
If $a=+\infty$, take $k_j=1$ and a sequence of integers
$N_j\geq2$ with $N_j\to+\infty$.

For each $j\geq1$, write $W^j=W(N_j,k_j)$ and equip
$W^j$ with a compatible metric $d_j$ such that
$\diam(W^j,d_j)\leq1$. Consider disjoint copies of these
spaces, add a point $*$, and define
$$
W_a=\{*\}\sqcup\left(\bigsqcup_{j=1}^{\infty}W^j\right).
$$ 
\begin{figure}[H]
    \centering
    `
     \includegraphics[scale=0.28]{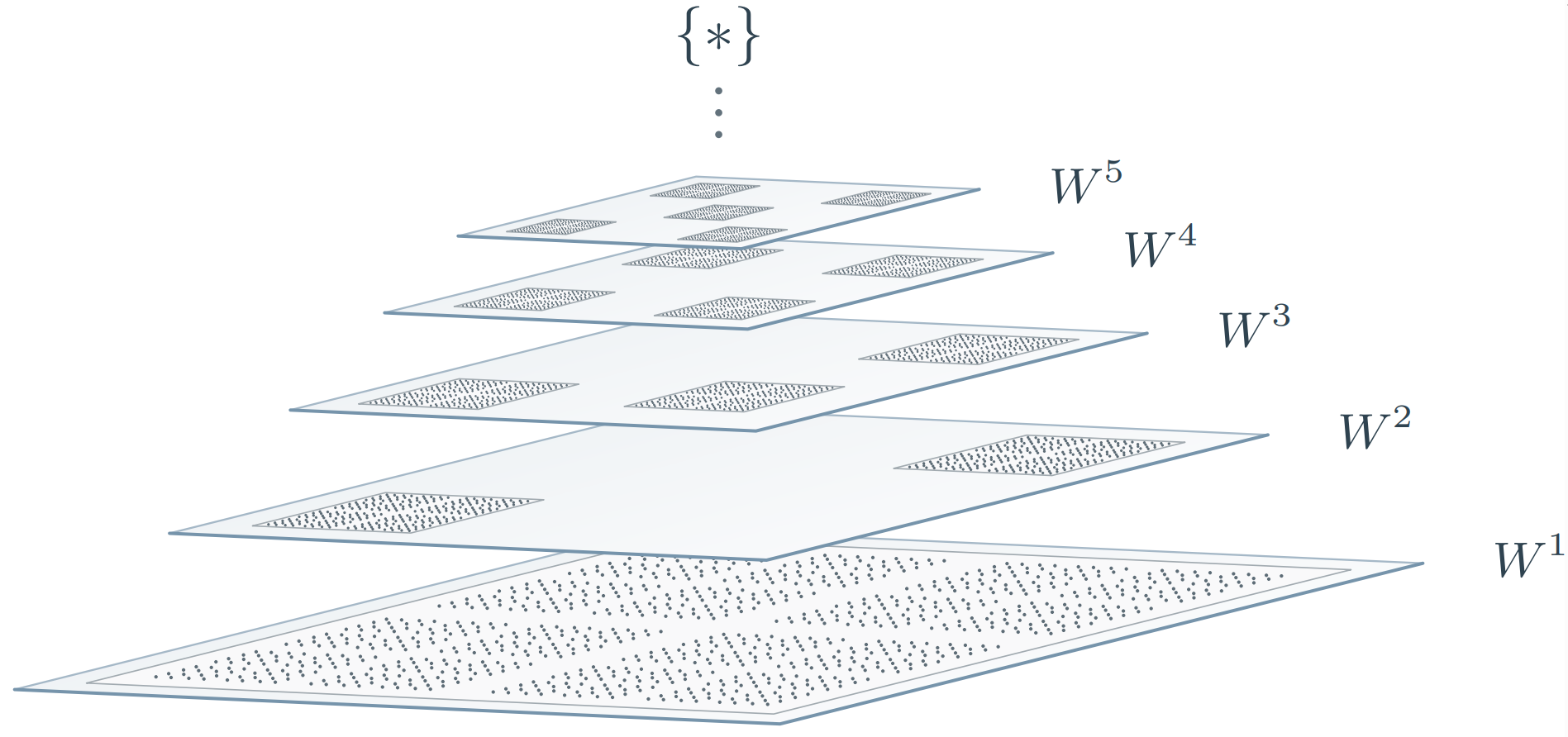}
    \caption{Schematic representation of $W_a$,
    illustrated with $k_j=j$.}
    \label{fig:Wa}
\end{figure}

Define a distance $d$ on $W_a$ by
$$
d(*,*)=0,
\qquad
d(*,x)=d(x,*)=2^{-j}
\quad(x\in W^j),
$$
and, for $x,y\neq *$,
$$
d(x,y)=
\begin{cases}
2^{-j}d_j(x,y),
&x,y\in W^j,\\
2^{-j}+2^{-\ell},
&x\in W^j,\ y\in W^\ell,\ j\neq\ell.
\end{cases}
$$
These formulas define a metric on $W_a$. The space $W_a$ is compact. Indeed, if a sequence
visits infinitely many sets $W^j$, it has a subsequence
$(x_{n_m})$ such that $x_{n_m}\in W^{j_m}$ and
$j_m\to+\infty$. Since
$d(x_{n_m},*)=2^{-j_m}\to0$, this subsequence converges to $*$.
If the sequence is contained in the union of finitely many
sets $W^j$ and the point $*$, compactness of that union
yields a convergent subsequence.

Define $T_a(*)=*$ and
$$
T_a|_{W^j}=S_{N_j,k_j}
\qquad(j\geq1).
$$
Each $W^j$ is open and closed in $W_a$, so $T_a$
and its inverse are continuous on each $W^j$. They are also
continuous at $*$, since they preserve each $W^j$, all of whose
points lie at distance $2^{-j}$ from $*$.
Thus, $T_a$ is a homeomorphism.

Since each $W^j$ is closed and invariant,
$$
h_{\rm top}(T_a)
\geq\sup_{j\geq1}h_{\rm top}(S_{N_j,k_j})
=\sup_{j\geq1}\frac{\log N_j}{k_j}
=a.
$$
This completes the proof when $a=+\infty$.
Thus, suppose that $a<+\infty$.

To prove the reverse inequality, fix $\varepsilon>0$
and choose $J\geq2$ such that $2^{1-J}<\varepsilon$.
The union
$$
R_J=\{*\}\cup\bigcup_{j\geq J}W^j
$$
satisfies
$$
d(T_a^q(x),T_a^q(y))
\leq2^{1-J}<\varepsilon
\qquad(x,y\in R_J,\ q\geq0).
$$
Let $s_m(F,\varepsilon)$ denote the maximum cardinality of an
$(m,\varepsilon)$-separated set for $F$. Since an
$(m,\varepsilon)$-separated set contains at most one point
of $R_J$, we have
$$
s_m(T_a,\varepsilon)
\leq
1+\sum_{j=1}^{J-1}s_m(S_{N_j,k_j},\varepsilon)\quad\Rightarrow\quad s_m(T_a,\varepsilon)\leq J\max_{1\leq j<J}s_m(S_{N_j,k_j},\varepsilon),
$$
where each $W^j$ is equipped with the metric induced by $W_a$.
Write
$$
h(F,\varepsilon)
=
\limsup_{m\to\infty}\frac1m\log s_m(F,\varepsilon).
$$
Since the preceding sum has finitely many terms,
$$
h(T_a,\varepsilon)
\leq
\max_{1\leq j<J}h(S_{N_j,k_j},\varepsilon)
\leq
\sup_{j\geq1}h_{\rm top}(S_{N_j,k_j})
=a.
$$
Taking the supremum over $\varepsilon>0$, we obtain
$h_{\rm top}(T_a)\leq a$. Consequently,
$h_{\rm top}(T_a)=a$.
\end{proof}

The choice of the limiting cycles and the cycles associated
with ergodic measures relies on the following fact.

\begin{lemma}
\label{lem:recurrentes}
Let $f:X\to X$ be a continuous map of a compact metric
space, and let $K\subset X$ be a nonempty compact set with
$f(K)\subset K$. There exists at least one ergodic
$f$-invariant measure $\mu$ with $\mu(K)=1$. Moreover, for each
ergodic measure $\mu$ with $\mu(K)=1$, there is a point
$\xi\in K$ such that
$$
f^{n_j}(\xi)\longrightarrow\xi
\quad\text{for some }n_j\longrightarrow\infty,
\qquad
\frac1n\sum_{k=0}^{n-1}\delta_{f^k(\xi)}
\xrightarrow[n\to\infty]{d_{\rm BL}}\mu.
$$
Consequently, given $a,\eta>0$, one can choose $p\geq1$
such that
$$
d_X(f^p(\xi),\xi)<a,
\qquad
d_{\rm BL}\!\left(
\frac1p\sum_{k=0}^{p-1}\delta_{f^k(\xi)},\mu
\right)<\eta.
$$
\end{lemma}

\begin{proof}
Take $x\in K$ and consider the measures
$$
\nu_N=\frac1N\sum_{k=0}^{N-1}\delta_{f^k(x)}.
$$
By compactness of $\mathbb P(K)$, some subsequence
converges weakly to a measure $\nu\in\mathbb P(K)$.
For every $\varphi\in C(K)$, we have
$$
\int\varphi\circ f\,d\nu_N-\int\varphi\,d\nu_N
=
\frac{\varphi(f^N(x))-\varphi(x)}{N}
\longrightarrow0.
$$
Thus, $\nu$ is invariant. Its ergodic decomposition
yields an ergodic measure supported on $K$.

Now fix any ergodic measure $\mu$ with $\mu(K)=1$.
Since $K$ is a compact metric space, $C(K)$ has a countable
family that is dense in the uniform norm. Applying Birkhoff's ergodic
theorem to this family and then using uniform
approximation shows that, for $\mu$-almost every $\xi\in K$,
$$
\frac1n\sum_{k=0}^{n-1}\varphi(f^k(\xi))
\longrightarrow\int_K\varphi\,d\mu
\qquad\text{for every }\varphi\in C(K).
$$
In particular, the empirical measures of such a point converge weakly
to $\mu$ and hence converge in $d_{\rm BL}$.

We show that $\xi$ can also be chosen to be recurrent.
Let $\{U_m:m\geq1\}$ be a countable base of open sets in $K$.
By Poincaré's recurrence theorem, for each $m$,
$\mu$-almost every point of $U_m$ returns to $U_m$ infinitely
often. Since the base is countable, we can choose a generic
point $\xi\in K$ satisfying this property for every
basic open set containing it.
For each $j\geq1$, take $U_{m_j}$ such that
$$
\xi\in U_{m_j}\subset B_K(\xi,1/j).
$$
Choose successive return times
$0<n_1<n_2<\cdots$ with
$f^{n_j}(\xi)\in U_{m_j}$. Then
$f^{n_j}(\xi)\to\xi$.

Finally, since $n_j\to\infty$, convergence of the
empirical measures allows us to choose $p=n_j$ sufficiently
large that the two inequalities in the statement
hold simultaneously.
\end{proof}

The following lemma combines the construction of the countable model
with the estimates needed to subsequently pass
to the product with the auxiliary system.

\begin{lemma}
\label{lem:modelo_numerable}
Let $f:X\to X$ be a surjective continuous map of a
compact metric space. For every $\delta>0$, there exist a
countable compact metric space $(Y_0,d_0)$, a homeomorphism
$g_0:Y_0\to Y_0$ with $h_{\rm top}(g_0)=0$, and
$\delta$-isometries $q_0:Y_0\to X$ and $s_0:X\to Y_0$,
with $q_0$ continuous, such that
$\Delta_{\rm BL}^{q_0}(g_0,f)<\delta$ and
$$
\sup_{y\in Y_0}
d_X\bigl(q_0(g_0(y)),f(q_0(y))\bigr)<\delta
\quad\text{and}\quad
\sup_{x\in X}
d_0\bigl(g_0(s_0(x)),s_0(f(x))\bigr)<\delta.
$$
\end{lemma}

\begin{proof}
Define
$$
\omega_f(\eta)
=
\sup\{d_X(f(x),f(x')):d_X(x,x')\leq\eta\}.
$$
Since $f$ is uniformly continuous, we can choose
$\theta>0$ such that
\begin{equation}\label{eq:A_eleccion_theta}
3\theta+\omega_f(\theta)<\delta.
\end{equation}

\medskip
\noindent\emph{Construction of the space and the dynamics.}
Let $\{x_1,\ldots,x_m\}$ be a $\theta$-net in $X$.
For each $i\in\{1,\ldots,m\}$, the surjectivity of $f$
allows us to choose a bi-infinite orbit
$(x_{i,n})_{n\in\mathbb Z}$ such that
$$
x_{i,0}=x_i,
\qquad
f(x_{i,n})=x_{i,n+1}
\quad(n\in\mathbb Z).
$$
Define the limit sets
$$
A_i^-=
\bigcap_{L\geq1}
\overline{\{x_{i,n}:n\leq-L\}},
\qquad
A_i^+=
\bigcap_{L\geq1}
\overline{\{x_{i,n}:n\geq L\}}.
$$
They are nonempty compact sets and satisfy
$$
f(A_i^\pm)=A_i^\pm.
$$
Indeed, continuity of $f$ implies
$f(A_i^\pm)\subset A_i^\pm$. For the reverse inclusion,
if $z\in A_i^\pm$, there is a sequence
$n_j\to\pm\infty$ such that $x_{i,n_j}\to z$.
By compactness, a subsequence of $(x_{i,n_j-1})$
converges to a point $w\in A_i^\pm$, and continuity
of $f$ gives $f(w)=z$.

By Lemma~\ref{lem:recurrentes}, for each sign we choose
a measure $\mu_i^\pm\in\mathbb E_f(X)$ supported
on $A_i^\pm$ and a recurrent point $\xi_i^\pm\in A_i^\pm$
that is generic for $\mu_i^\pm$.
We also choose $p_i^\pm\geq1$ such that
\begin{align}
d_X\bigl(f^{p_i^\pm}(\xi_i^\pm),\xi_i^\pm\bigr)
&<\theta,
\label{eq:A_retorno}\\
d_{\rm BL}
\bigl(\mathsf e_{p_i^\pm}^f(\xi_i^\pm),\mu_i^\pm\bigr)
&<\theta.
\label{eq:A_empirica}
\end{align}
Since $f(\xi_i^-)\in A_i^-$ and $\xi_i^+\in A_i^+$, we can choose
$L_i^-\geq1$ and $L_i^+\geq2$ so that
\begin{equation}\label{eq:A_empalmes}
d_X(x_{i,-L_i^-+1},f(\xi_i^-))<\theta,
\qquad
d_X(x_{i,L_i^+},\xi_i^+)<\theta.
\end{equation}

For each $i$, let
$$
\Lambda_i=\{u_{i,n}:n\in\mathbb Z\}
$$
be an abstract copy of $\mathbb Z$. The sets
$\Lambda_i$ are taken to be pairwise disjoint and disjoint from $X$.
Take two further finite sets
$$
C_i^-=\{c_{i,0}^-,\ldots,c_{i,p_i^--1}^-\},
\qquad
C_i^+=\{c_{i,0}^+,\ldots,c_{i,p_i^+-1}^+\},
$$
disjoint from one another and from the sets $\Lambda_i$, and define
$$
Y_i=\Lambda_i\sqcup C_i^-\sqcup C_i^+.
$$

For $n\in\mathbb Z$ and $p\geq1$, denote by $[n]_p$
the unique element of $\{0,\ldots,p-1\}$ congruent to $n$
modulo $p$. Declare all points $u_{i,n}$ to be isolated.
For $0\leq k<p_i^-$, a neighborhood base at $c_{i,k}^-$
is given by
$$
\{c_{i,k}^-\}\cup
\{u_{i,n}:n\leq-N,\ [n+L_i^-]_{p_i^-}=k\},
\qquad N>L_i^-,
$$
and for $0\leq k<p_i^+$, a neighborhood base at $c_{i,k}^+$
is given by
$$
\{c_{i,k}^+\}\cup
\{u_{i,n}:n\geq N,\ [n-L_i^+]_{p_i^+}=k\},
\qquad N>L_i^+,
$$
where $N$ ranges over the indicated integers.

With this topology, $Y_i$ is homeomorphic to a finite
topological sum of spaces consisting of a convergent sequence
and its limit, together with a finite discrete space.
Thus, $Y_i$ is compact, countable, and metrizable.

Define $g_0:Y_i\to Y_i$ by
$$
g_0(u_{i,n})=u_{i,n+1},
\qquad
g_0(c_{i,k}^\pm)
=c_{i,[k+1]_{p_i^\pm}}^\pm.
$$
The shift of indices respects the congruence classes
modulo $p_i^\pm$, so $g_0$ and $g_0^{-1}$
preserve the preceding convergences.
Consequently, $g_0$ is a homeomorphism.

\medskip
\noindent\emph{Map to the original space.}
Define $q_0:Y_i\to X$ by
$$
q_0(u_{i,n})=
\begin{cases}
f^{[n+L_i^-]_{p_i^-}}(\xi_i^-),
&n\leq-L_i^-,\\[1mm]
x_{i,n},
&-L_i^-<n<L_i^+,\\[1mm]
f^{[n-L_i^+]_{p_i^+}}(\xi_i^+),
&n\geq L_i^+,
\end{cases}\;\text{ and }\qquad q_0(c_{i,k}^\pm)=f^k(\xi_i^\pm).
$$

\begin{figure}[H]
    \centering
    \includegraphics[width=1\linewidth]{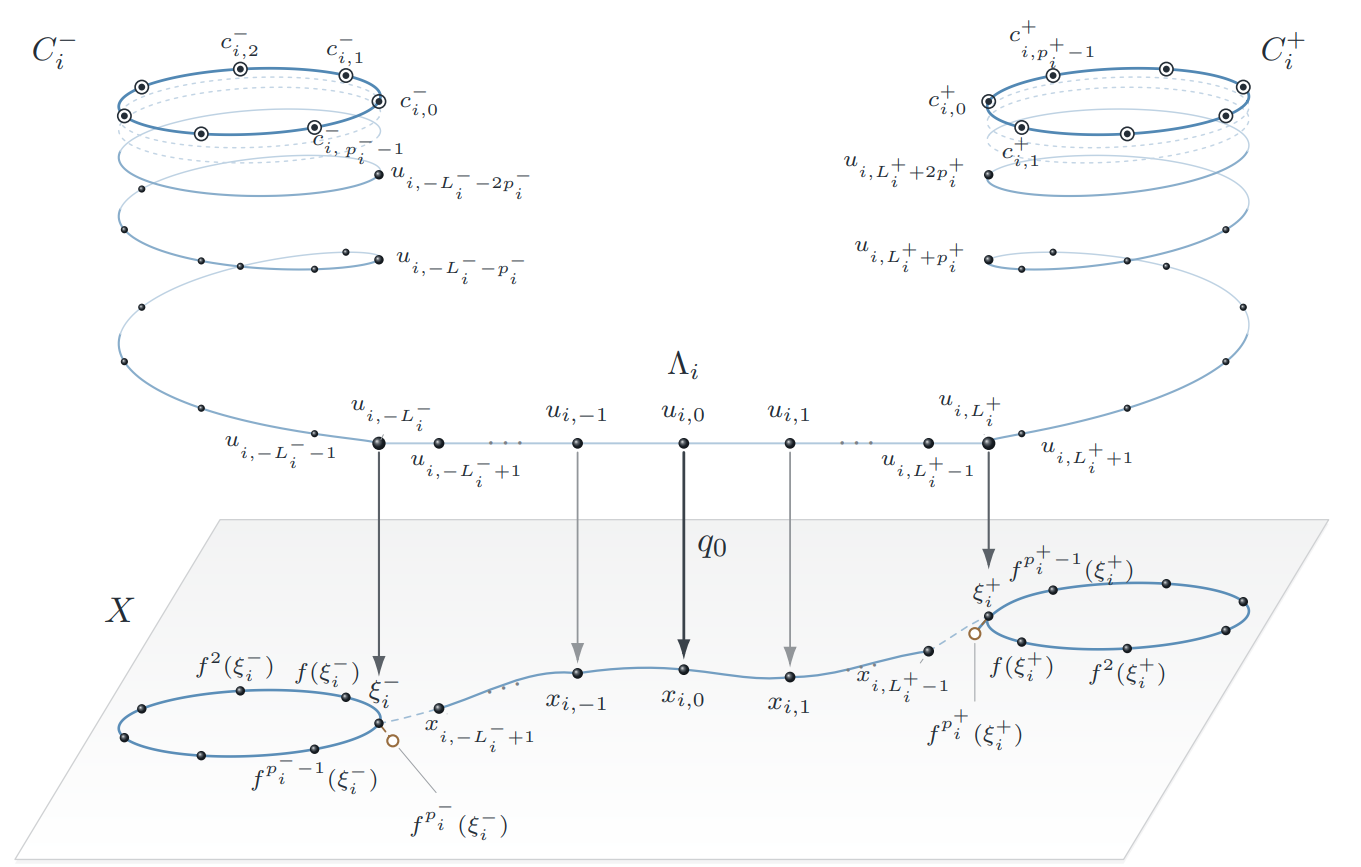}
    \caption{The map $q_0|_{Y_i}$.}
    \label{fig:Wa}
\end{figure}

The map $q_0$ is continuous. At the points of $\Lambda_i$,
this is immediate, since they are isolated. At each point
$c_{i,k}^\pm$, the map is constant and equal to
$f^k(\xi_i^\pm)$ on a sufficiently small basic neighborhood.

Moreover,
\begin{equation}\label{eq:A_defecto}
d_X\bigl(q_0(g_0(y)),f(q_0(y))\bigr)<2\theta
\qquad(y\in Y_i).
\end{equation}
If $-L_i^-<n<L_i^+-1$, the error at $u_{i,n}$ is zero.
At the transition
$u_{i,-L_i^-}\mapsto u_{i,-L_i^-+1}$,
the error is $d_X(x_{i,-L_i^-+1},f(\xi_i^-))<\theta$.
At the transition
$u_{i,L_i^+-1}\mapsto u_{i,L_i^+}$,
the error is $d_X(x_{i,L_i^+},\xi_i^+)<\theta$.
At all other points of the tails and the cycles
$C_i^\pm$, the error is zero, except when the index
passes from $p_i^\pm-1$ to $0$, in which case it is less
than $\theta$ by \eqref{eq:A_retorno}.

Finally, $q_0(u_{i,0})=x_i$.
Since $\{x_1,\ldots,x_m\}$ is a $\theta$-net,
the image of $\bigsqcup_{i=1}^mY_i$ is
$\theta$-dense in $X$.

\medskip
\noindent\emph{Control of measures and entropy.}
The set $\mathbb E_f(X)$ is nonempty by
Lemma~\ref{lem:recurrentes} and is totally bounded,
since it is contained in the compact metric space
$\mathbb P(X)$. Choose a finite $\theta$-net
$$
\{\beta_1,\ldots,\beta_\ell\}
\subset\mathbb E_f(X)
$$
with respect to $d_{\rm BL}$.
For each $j$, Lemma~\ref{lem:recurrentes}
provides a recurrent point $\eta_j$ that is generic
for $\beta_j$ and an integer $m_j\geq1$ such that
$$
d_X\bigl(f^{m_j}(\eta_j),\eta_j\bigr)<\theta,
\qquad
d_{\rm BL}
\bigl(\mathsf e_{m_j}^f(\eta_j),\beta_j\bigr)<\theta.
$$

Take finite sets
$$
B_j=\{v_{j,0},\ldots,v_{j,m_j-1}\},
$$
pairwise disjoint and disjoint from all the $Y_i$.
Extend $g_0$ and $q_0$ by
$$
g_0(v_{j,k})=v_{j,[k+1]_{m_j}},
\qquad
q_0(v_{j,k})=f^k(\eta_j).
$$
Estimate \eqref{eq:A_defecto} remains valid
on each $B_j$, including the step that closes the cycle.

Define
$$
Y_0=
\left(\bigsqcup_{i=1}^mY_i\right)
\sqcup
\left(\bigsqcup_{j=1}^{\ell}B_j\right)
$$
with the disjoint union topology.
Since it is a finite sum of compact metrizable spaces,
$Y_0$ is compact, metrizable, and countable.
The preceding definitions yield a homeomorphism
$g_0:Y_0\to Y_0$ and a continuous map $q_0:Y_0\to X$
whose image is $\theta$-dense and which satisfies
\eqref{eq:A_defecto} throughout $Y_0$.

We now determine the invariant measures of this model.
Let $\nu\in\mathcal M_{g_0}(Y_0)$.
For each $i$, invariance implies
$$
\nu(\{u_{i,n}\})
=\nu(\{u_{i,n+1}\})
\qquad(n\in\mathbb Z).
$$
All points of this infinite orbit therefore have
the same mass, which must be zero.
Since $\bigcup_{i=1}^m\Lambda_i$ is countable, every invariant
measure is concentrated on the finite union of the cycles
$C_i^\pm$ and $B_j$.

If $\nu_C$ denotes the uniform measure on a cycle $C$,
then
\begin{equation}\label{eq:A_medidas-ciclos}
\begin{aligned}
\mathbb E_{g_0}(Y_0)
&=\{\nu_C:C\text{ is a cycle in }Y_0\},\\
\mathcal M_{g_0}(Y_0)
&=\operatorname{conv}\mathbb E_{g_0}(Y_0).
\end{aligned}
\end{equation}
All these measures have zero measure-theoretic entropy.
By the variational principle \cite{walters1982},
\begin{equation}\label{eq:A_h-cero}
h_{\rm top}(g_0)=0.
\end{equation}

For the cycles $C_i^\pm$,
$$
q_{0*}\nu_{C_i^\pm}
=\mathsf e_{p_i^\pm}^f(\xi_i^\pm),
$$
and by \eqref{eq:A_empirica},
$$
d_{\rm BL}
\bigl(q_{0*}\nu_{C_i^\pm},\mu_i^\pm\bigr)<\theta.
$$
Likewise,
$$
q_{0*}\nu_{B_j}
=\mathsf e_{m_j}^f(\eta_j),
\qquad
d_{\rm BL}
\bigl(q_{0*}\nu_{B_j},\beta_j\bigr)<\theta.
$$
Every measure in $\mathbb E_f(X)$ lies at distance less than $\theta$
from some $\beta_j$, and hence at distance less than $2\theta$
from $q_{0*}\nu_{B_j}$. Thus,
\begin{equation}\label{eq:A_ergodicas}
d_H^{\rm BL}
\bigl(q_{0*}\mathbb E_{g_0}(Y_0),\mathbb E_f(X)\bigr)
\leq2\theta.
\end{equation}

Since
$$
\mathcal M_f(X)
=\overline{\operatorname{conv}}\mathbb E_f(X),
$$
the convexity inequality for $d_{\rm BL}$ preserves
the same bound when taking convex combinations
and closures. Moreover,
$q_{0*}\mathcal M_{g_0}(Y_0)$ is the convex hull
of a finite set and is therefore compact.
By \eqref{eq:A_medidas-ciclos}, we obtain
\begin{equation}\label{eq:A_invariantes}
d_H^{\rm BL}
\bigl(q_{0*}\mathcal M_{g_0}(Y_0),\mathcal M_f(X)\bigr)
\leq2\theta.
\end{equation}
In particular,
$$
\Delta_{\rm BL}^{q_0}(g_0,f)\leq2\theta.
$$

\medskip
\noindent\emph{Choice of the metric.}
Choose a compatible metric $\rho$ on $Y_0$
with diameter at most $1$, and define
\begin{equation}\label{eq:A_metrica}
d_0(u,v)
=d_X\bigl(q_0(u),q_0(v)\bigr)+\theta\rho(u,v).
\end{equation}
This metric induces the topology of $Y_0$:
the inequality $d_0\geq\theta\rho$ gives one
implication between the notions of convergence, and continuity
of $q_0$ gives the other.
Moreover,
\begin{equation}\label{eq1}
\dis(q_0)\leq\theta,
\qquad
\sup_{x\in X}d_X(x,q_0(Y_0))<\theta.
\end{equation}

For each $x\in X$, choose an index
$i(x)\in\{1,\ldots,m\}$ such that
$d_X(x,x_{i(x)})<\theta$ and define
$$
s_0(x)=u_{i(x),0}.
$$
The map $s_0:X\to Y_0$ need not be continuous.
For $x,x'\in X$, we have
$$
\begin{aligned}
\left|
d_0\bigl(s_0(x),s_0(x')\bigr)-d_X(x,x')
\right|
&\leq d_X(x_{i(x)},x)
+d_X(x_{i(x')},x')+\theta\\
&\leq3\theta.
\end{aligned}
$$
On the other hand, for each $u\in Y_0$,
$$
d_0\bigl(u,s_0(q_0(u))\bigr)\leq2\theta.
$$
Consequently,
\begin{equation}\label{eq2}
\dis(s_0)\leq3\theta,
\qquad
\sup_{u\in Y_0}d_0(u,s_0(X))\leq2\theta.
\end{equation}

Finally,
$$
\sup_{x\in X}
d_0\bigl(g_0(s_0(x)),s_0(f(x))\bigr)
\leq\omega_f(\theta)+2\theta.
$$
Indeed, fix $x\in X$ and set
$i=i(x)$ and $j=i(f(x))$.
Since $L_i^+\geq2$, we have
$$
q_0\bigl(g_0(s_0(x))\bigr)=f(x_i),
\qquad
q_0\bigl(s_0(f(x))\bigr)=x_j.
$$
Therefore,
$$
d_X(f(x_i),x_j)
\leq\omega_f(\theta)+\theta.
$$
The term $\theta\rho$ contributes at most
$\theta$, proving the estimate. By \eqref{eq:A_defecto},
$$
\sup_{y\in Y_0}
d_X\bigl(q_0(g_0(y)),f(q_0(y))\bigr)\leq2\theta.
$$
All the metric and dynamical bounds obtained,
as well as the bound on $\Delta_{\rm BL}^{q_0}(g_0,f)$,
are less than $\delta$ by
\eqref{eq:A_eleccion_theta}. This completes the proof.
\end{proof}

\begin{proof}[Proof of Theorem~\ref{thm:A}]
\noindent\emph{Implication $(1)\Rightarrow(2)$.}
Suppose that $f$ is surjective and fix
$\varepsilon>0$ and $t\in[0,+\infty]$.
Apply Lemma~\ref{lem:modelo_numerable} with
$\delta=\varepsilon/4$, and let
$(Y_0,d_0)$, $g_0$, $q_0$, and $s_0$ be the resulting objects.

Take $(W_t,T_t,*)$ from
Lemma~\ref{lem:sistema_auxiliar_entropia}, with a metric
$d_{W_t}$ of diameter at most $1$. Define
$$
Y=Y_0\times W_t,
\qquad
g=g_0\times T_t,
$$
and the maps
$$
q(u,w)=q_0(u),
\qquad
s(x)=(s_0(x),*).
$$
Equip $Y$ with the metric
$$
d_Y\bigl((u,w),(v,w')\bigr)
=d_0(u,v)+\delta d_{W_t}(w,w').
$$
The space $Y$ is compact, $g$ is a homeomorphism,
and $q$ is continuous.

Estimates (\ref{eq1}) and (\ref{eq2}) in Lemma~\ref{lem:modelo_numerable} and the bound on the diameter
of $W_t$ give
$$
\dis(q)\leq2\delta,
\qquad
\sup_{x\in X}d_X(x,q(Y))<\delta,
$$
and
$$
\dis(s)<\delta,
\qquad
\sup_{y\in Y}d_Y(y,s(X))\leq3\delta.
$$
Since $T_t(*)=*$, we also have
$$
\sup_{y\in Y}
d_X\bigl(q(g(y)),f(q(y))\bigr)<\delta,
$$
$$
\sup_{x\in X}
d_Y\bigl(g(s(x)),s(f(x))\bigr)<\delta.
$$
Thus, $q$ and $s$ are $3\delta$-isometries
satisfying the conditions in the definition of $d_{GH^0}$.
Consequently,
$$
d_{GH^0}(f,g)\leq3\delta<\varepsilon,
$$
and $q$ is a continuous $\varepsilon$-isometry.

By the product formula for entropy,
$$
h_{\rm top}(g)
=h_{\rm top}(g_0)+h_{\rm top}(T_t)=t.
$$
If $t=0$, take $W_0=\{*\}$, so $Y$ is countable.

Now let $\pi:Y\to Y_0$ be the projection onto the first
factor. This is a factor map, whereas
$$
\iota:Y_0\to Y,\qquad \iota(u)=(u,*),
$$
is an equivariant embedding with
$\pi\circ\iota=\operatorname{id}_{Y_0}$.
The projection of an invariant or ergodic measure preserves
the corresponding property. Conversely, $\iota_*\nu$
is an invariant or ergodic lift of each measure
$\nu$ of the same type. Therefore,
$$
\pi_*\mathcal M_g(Y)=\mathcal M_{g_0}(Y_0),
\qquad
\pi_*\mathbb E_g(Y)=\mathbb E_{g_0}(Y_0).
$$
Since $q=q_0\circ\pi$, we obtain
$$
\Delta_{\rm BL}^{q}(g,f)
=\Delta_{\rm BL}^{q_0}(g_0,f)
<\delta<\varepsilon.
$$
This proves the first implication.

\medskip
\noindent\emph{Implication $(2)\Rightarrow(1)$.}
Suppose that (2) holds and that $f$ is not surjective.
Since $f(X)$ is compact, there is $x_*\in X$ such that
$$
\kappa=d_X(x_*,f(X))>0.
$$
Applying (2) with $t=0$ and $\varepsilon=\kappa/3$,
and using the definition of $d_{GH^0}$, we obtain
a homeomorphism $g:Y\to Y$ and a map
$j:Y\to X$ whose image is $\kappa/3$-dense and which satisfies
$$
\sup_{y\in Y}
d_X\bigl(j(g(y)),f(j(y))\bigr)<\kappa/3.
$$
Choose $y\in Y$ with $d_X(x_*,j(y))<\kappa/3$.
By surjectivity of $g$, there is $z\in Y$
such that $g(z)=y$. Then
$$
\begin{aligned}
\kappa
&\leq d_X\bigl(x_*,f(j(z))\bigr)\\
&\leq d_X(x_*,j(y))
+d_X\bigl(j(g(z)),f(j(z))\bigr)
<\frac{2\kappa}{3},
\end{aligned}
$$
a contradiction. Thus, $f$ is surjective.
\end{proof}

In Theorem~\ref{thm:A}, when $t=0$, the approximating space
can be chosen to be compact and countable. For topologically
transitive homeomorphisms, the construction in
\cite[proof of Theorem~3.1.3]{cubas2018} yields
approximations in the $GH^0$ distance by dynamical systems on
finite spaces; see also \cite{jung2019} for results related
to approximation by periodic orbits. In general, however,
the approximating space cannot be required to be finite.
The following example shows this obstruction, even for
a homeomorphism with positive entropy.

\begin{example}
\label{ex:finito_no_basta}
Let $A:\mathbb T^2\to\mathbb T^2$ be the toral automorphism
induced by the matrix
$$
\begin{pmatrix}
2&1\\
1&1
\end{pmatrix}.
$$
On $X=\mathbb T^2\times[0,1]$, consider the metric
$$
d_X\bigl((u,s),(v,t)\bigr)
=\max\{d_{\mathbb T^2}(u,v),|s-t|\},
$$
and define the homeomorphism $F(u,s)=(A(u),s^2)$. Since the map $s\mapsto s^2$ has zero topological entropy,
$$
h_{\rm top}(F)
=h_{\rm top}(A)
=\log\left(\frac{3+\sqrt5}{2}\right)>0.
$$

We claim that $d_{GH^0}(F,g)\geq10^{-2}$ for every homeomorphism $g:Y\to Y$ with $Y$ finite.
Indeed, suppose otherwise. Then there exist
$0<\delta<10^{-2}$ and a $\delta$-isometry $j:Y\to X$
such that
$$
d_X\bigl(j(g(y)),F(j(y))\bigr)<\delta
\qquad(y\in Y).
$$
Since $j(Y)$ is $\delta$-dense, for any fixed $u\in\mathbb T^2$
there is $y_0\in Y$ such that $d_X\bigl(j(y_0),(u,1/2)\bigr)<\delta$. Write $j(g^n(y_0))=(u_n,s_n)$. Then $0.49<s_0<0.51$ and, for every $n\geq0$,
$$
|s_{n+1}-s_n^2|<\delta.
$$
Therefore, $s_1<0.51^2+0.01<0.3$. Moreover, if $s_n<0.3$, then
$$
s_{n+1}<0.3^2+0.01<0.3.
$$
Thus, $s_n<0.3$ for every $n\geq1$.
Since $Y$ is finite and $g$ is a homeomorphism,
$y_0$ is periodic. There is $N\geq1$ such that
$g^N(y_0)=y_0$, and hence $s_N=s_0$.
This contradicts $s_N<0.3<s_0$.

We conclude that no homeomorphism of a finite space
can approximate $F$ within a distance less than $10^{-2}$.
Nevertheless, Theorem~\ref{thm:A} provides countable
compact zero-entropy models arbitrarily
close to $F$.
The obstruction comes from the second coordinate;
the torus factor allows us to exhibit it in a system
with positive entropy.
\end{example}
\section{Proof of Theorem~\ref{thm:B} }\label{sec:rigidez}

We prove Theorem~\ref{thm:B} by establishing that sufficiently
close perturbations within the fixed geometric class have the
reference system as a factor. Surjectivity of the factor
map follows from Mitsuishi and Yamaguchi's homotopy stability
theorem for noncollapsing Alexandrov spaces, applied to the
continuous approximation provided by $GH$-stability. A homology
argument treats all dimensions uniformly.
The entropy inequality and the equalities concerning measures follow from
the existence of the factor. We begin with the following measure-lifting lemma, which will be used at the end of the proof.

\begin{lemma}
\label{lem:levantamiento_factor}
Let $f:X\to X$ and $g:Y\to Y$ be continuous maps. If
$\pi:Y\to X$ is continuous and surjective, and $\pi\circ g=f\circ\pi$,
then
$$
\pi_*\mathcal M_g(Y)=\mathcal M_f(X),
\qquad
\pi_*\mathbb{E}_g(Y)=\mathbb{E}_f(X).
$$
\end{lemma}

\begin{proof}
The semiconjugacy relation implies
$\pi_*\mathcal M_g(Y)\subset\mathcal M_f(X)$. Fix
$\mu\in\mathcal M_f(X)$. Since $\pi$ is surjective, the map
$$
C(X)\to C(Y),
\qquad
\varphi\mapsto\varphi\circ\pi,
$$
is a linear isometry. The functional
$$
\varphi\circ\pi\mapsto\int_X\varphi\,d\mu
$$
has norm one and takes the value one on the constant function. The
Hahn--Banach theorem and the Riesz representation theorem yield a measure
$\lambda\in\mathbb{P}(Y)$ with $\pi_*\lambda=\mu$. Indeed, a norm-one extension
that takes the value one on the constant function is positive. Now consider $\lambda_m=\dfrac1m\sum_{k=0}^{m-1}g_*^k\lambda$. Every accumulation point $\nu$ of $(\lambda_m)$ is $g$-invariant, and
$$
\pi_*\lambda_m
=\frac1m\sum_{k=0}^{m-1}f_*^k\mu=\mu.
$$
Therefore, $\pi_*\nu=\mu$, proving the first equality.

An ergodic measure projects to an ergodic measure under a factor map, so
$\pi_*\mathbb{E}_g(Y)\subset\mathbb{E}_f(X)$. Let
$\mu\in\mathbb{E}_f(X)$ and choose
$\nu\in\mathcal M_g(Y)$ with $\pi_*\nu=\mu$. If $\tau_\nu$ is the
ergodic decomposition of $\nu$, then
$$
\mu=\int_{\mathbb{E}_g(Y)}\pi_*\xi\,d\tau_\nu(\xi).
$$
Since $\mu$ is ergodic, its ergodic decomposition is the point mass
at $\mu$. By uniqueness of the ergodic decomposition applied to the
preceding equality, $\pi_*\xi=\mu$ for $\tau_\nu$-almost every $\xi$.
In particular, $\mu$ has an ergodic lift.
\end{proof}

We next state the homotopy stability result
of Mitsuishi and Yamaguchi. To state it, given $n\geq1$ and $D,v>0$,
denote by $\mathcal A(n,D,v)$ the class of compact Alexandrov
spaces of dimension $n$, with curvature bounded below
by $-1$, diameter at most $D$, and $n$-dimensional volume
at least $v$. In this setting, an $\varepsilon$-approximation
is a map, not necessarily continuous, with distortion
less than $\varepsilon$ and $\varepsilon$-dense image.
We present the result in the notation to be used
in the proof.

\begin{theorem}[Mitsuishi--Yamaguchi,
{\cite[Theorem~1.2]{mitsuishiYamaguchi2019}}]
\label{thm:MY}
For fixed $n,D,v$, there exists $\varepsilon_*>0$ such that,
if $M,M'\in\mathcal A(n,D,v)$,
$0<\varepsilon<\varepsilon_*$, and $d_{GH}(M,M')<\varepsilon$,
every $\varepsilon$-approximation $\theta:M\to M'$
admits a Lipschitz homotopy equivalence
$u:M\to M'$ satisfying
$$
\sup_{x\in M}d_{M'}\bigl(u(x),\theta(x)\bigr)
\leq\tau(\varepsilon),
\qquad
\tau(\varepsilon)\longrightarrow0
\quad\text{as }\varepsilon\longrightarrow0.
$$
\end{theorem}
We shall apply this result to show that the maps
provided by $GH$-stability are homotopic to
homotopy equivalences and, through a homology argument,
deduce their surjectivity.

\begin{proof}[Proof of Theorem~\ref{thm:B}]
We first prove that, for sufficiently close perturbations
within the class $\mathcal R(n,K,i_0,D)$, there is
a factor map onto $(X,f)$. Suppose, by
contradiction, that no $\delta>0$ has this property.

For each $j\geq1$, set $\varepsilon_j=1/j$.
By the $GH$-stability of $f$, there is $\eta_j>0$ such that
every homeomorphism $g:Y\to Y$ with
$d_{GH^0}(f,g)<\eta_j$ admits a
continuous $\varepsilon_j$-isometry $h:Y\to X$ satisfying
$f\circ h=h\circ g$. By our assumption, we can choose
$Y_j\in\mathcal R(n,K,i_0,D)$ and a homeomorphism
$g_j:Y_j\to Y_j$ such that
$$
d_{GH^0}(f,g_j)<\min\{\eta_j,1/j\},
$$
but $(X,f)$ is not a factor of $(Y_j,g_j)$.
Stability then provides a
continuous $\varepsilon_j$-isometry $h_j:Y_j\to X$ such that
$$
f\circ h_j=h_j\circ g_j,
\qquad
\dis(h_j)<\varepsilon_j,
\qquad
\sup_{x\in X}d_X(x,h_j(Y_j))<\varepsilon_j.
$$
Since $(X,f)$ is not a factor of $(Y_j,g_j)$,
the map $h_j$ cannot be surjective.

We now show that the volumes of manifolds in $\mathcal R(n,K,i_0,D)$
are uniformly bounded below.
If $n\geq2$, fix
$$
r_0=
\min\left\{
\frac{i_0}{2},
\frac{\pi}{2\sqrt K}
\right\}.
$$
For every $M\in\mathcal R(n,K,i_0,D)$, we have
$r_0<\inj(M)$ and $\sec_M\leq K$.
By Günther's volume comparison theorem
\cite[Theorem~11.14]{lee2018riemannian},
for any $p\in M$,
$$
\operatorname{vol}_n(B_M(p,r_0))
\geq
\operatorname{vol}(\mathbb S^{n-1})
\int_0^{r_0}
\left(
\frac{\sin(\sqrt K\,s)}{\sqrt K}
\right)^{n-1}ds,
$$
where $\mathbb S^{n-1}$ is the unit sphere.
The right-hand side depends only on $n$, $K$,
and $i_0$ and is strictly positive. Since $\operatorname{vol}_n(M)
\geq\operatorname{vol}_n(B_M(p,r_0))$, we obtain a uniform lower bound on volume.\\

If $n=1$, each $M$ is a circle and $\operatorname{vol}_1(M)=2\inj(M)\geq2i_0$. Therefore, in every dimension there is a constant $v_0=v_0(n,K,i_0)>0$
such that
$$
\operatorname{vol}_n(M)\geq v_0
\qquad
\text{for every }M\in\mathcal R(n,K,i_0,D).
$$
Let $c=\max\{1,\sqrt K\}$, and denote by $cM$ the manifold
$M$ with its distance multiplied by $c$.
Under this rescaling,
$$
\diam(cM)=c\,\diam(M),
\qquad
\operatorname{vol}_n(cM)=c^n\operatorname{vol}_n(M).
$$
Moreover, for $n\geq2$,
$$
\sec_{cM}=c^{-2}\sec_M\geq-K/c^2\geq-1.
$$
In dimension one, circles also satisfy the curvature
lower bound in the sense of Alexandrov.
Consequently,
$$
cX,\ cY_j\in\mathcal A(n,cD,c^nv_0).
$$
With respect to the rescaled metrics, $h_j:cY_j\to cX$
satisfies
$$
\dis(h_j)<c\varepsilon_j,
\qquad
\sup_{x\in cX}d_{cX}(x,h_j(cY_j))<c\varepsilon_j.
$$
Thus, $h_j$ is a $c\varepsilon_j$-approximation.
The standard estimate relating approximations to the
Gromov--Hausdorff distance gives
$$
d_{GH}(cY_j,cX)\leq2c\varepsilon_j.
$$

Let $\varepsilon_*>0$ be the threshold in Theorem~\ref{thm:MY}
for the class $\mathcal A(n,cD,c^nv_0)$, and set $\alpha_j=4c\varepsilon_j$.
For all sufficiently large $j$,
$\alpha_j<\varepsilon_*$; moreover, $h_j$ is an
$\alpha_j$-approximation and
$$
d_{GH}(cY_j,cX)<\alpha_j.
$$
Apply Theorem~\ref{thm:MY} with
$$
M=cY_j,\qquad M'=cX,\qquad
\theta=h_j,\qquad \varepsilon=\alpha_j.
$$
We obtain a Lipschitz homotopy equivalence $u_j:cY_j\longrightarrow cX$ such that
$$
\sup_{y\in Y_j}
d_{cX}\bigl(u_j(y),h_j(y)\bigr)
\leq\tau(\alpha_j), \text{ with }\tau(\alpha_j)\to 0.
$$
Rescaling does not change the topologies, so
$u_j:Y_j\to X$ is also a homotopy equivalence
for the original metrics. Moreover,
$$
\sup_{y\in Y_j}
d_X\bigl(u_j(y),h_j(y)\bigr)
\leq\frac{\tau(\alpha_j)}{c}
\longrightarrow0.
$$
Since $\inj(X)\geq i_0$, for all sufficiently large $j$,
$$
d_X\bigl(u_j(y),h_j(y)\bigr)<i_0
\qquad\text{for every }y\in Y_j.
$$
Thus, $u_j(y)$ and $h_j(y)$ are joined by a unique
minimizing geodesic, which depends continuously on $y$.
These geodesics define a homotopy between $u_j$ and $h_j$.
Hence,
$$
h_j\simeq u_j.
$$

Since $u_j$ is a homotopy equivalence, the induced map
on singular homology
$$
(u_j)_*:
H_n(Y_j;\mathbb Z_2)
\longrightarrow H_n(X;\mathbb Z_2)
$$
is an isomorphism. Homotopy invariance of homology
implies
$$
(h_j)_*=(u_j)_*,
$$
so $(h_j)_*$ is also an isomorphism.
Moreover, since $X$ is a closed, connected manifold
of dimension $n$, $H_n(X;\mathbb Z_2)\cong\mathbb Z_2$; see \cite[Theorem~3.26]{hatcher2002}. However, $h_j$ is not surjective.
We can choose $p_j\in X\setminus h_j(Y_j)$, so
that $h_j$ factors as
$$
Y_j
\xrightarrow{\;\widehat h_j\;}
X\setminus\{p_j\}
\xrightarrow{\;\iota_j\;}
X,
$$
where $\iota_j$ is the inclusion.
The space $X\setminus\{p_j\}$ is a connected,
noncompact manifold of dimension $n$, and hence
$$
H_n(X\setminus\{p_j\};\mathbb Z_2)=0;
$$
see \cite[Proposition~3.29]{hatcher2002}.
Consequently,
$$
(h_j)_*
=(\iota_j)_*\circ(\widehat h_j)_*
=0,
$$
which contradicts the fact that $(h_j)_*$ is an isomorphism
onto the nonzero group $H_n(X;\mathbb Z_2)$.

This contradiction proves that there exists $\delta>0$ such that,
if $Y\in\mathcal R(n,K,i_0,D)$ and $g:Y\to Y$ is a
homeomorphism with $d_{GH^0}(f,g)<\delta$, there is a
continuous surjective map $h:Y\to X$ satisfying
$$
f\circ h=h\circ g.
$$
Thus, $(X,f)$ is a factor of $(Y,g)$.
Monotonicity of topological entropy under factors implies
$$
h_{\rm top}(g)\geq h_{\rm top}(f).
$$

Finally, applying Lemma~\ref{lem:levantamiento_factor}
to the same map $h$, we obtain
$$
h_*\mathcal M_g(Y)=\mathcal M_f(X),
\qquad
h_*\mathbb E_g(Y)=\mathbb E_f(X).
$$
\end{proof}
\begin{remark}
A continuous map with small distortion and an image that is dense
at the scale of the approximation need not be surjective.
For example, write $\mathbb S^1=\mathbb R/\mathbb Z$ with
its length metric, and denote by
$\mathbb S^1_{\varepsilon_j}$ the same circle with its distance
multiplied by $\varepsilon_j>0$. Consider
$$
M_j=N_j=\mathbb S^1\times\mathbb S^1_{\varepsilon_j},
\qquad \varepsilon_j\longrightarrow0,
$$
with the Riemannian product metric. These are flat tori,
that is, tori whose Gaussian curvature vanishes identically.
Equivalently, they are obtained by identifying opposite sides
of a Euclidean rectangle with side lengths $1$ and $\varepsilon_j$.

The maps
$$
h_j:M_j\to N_j,\qquad h_j(x,y)=(x,0),
$$
where $0$ denotes the class of zero in the second circle,
are continuous and not surjective. If
$a_j=\diam(\mathbb S^1_{\varepsilon_j})=\varepsilon_j/2$,
then
$$
\dis(h_j)\leq a_j,
\qquad
\sup_{z\in N_j}d_{N_j}(z,h_j(M_j))\leq a_j,
$$
and both quantities tend to zero. At the same time,
$$
\operatorname{vol}_2(M_j)=\varepsilon_j\longrightarrow0,
\qquad
\inj(M_j)=\frac12\min\{1,\varepsilon_j\}\longrightarrow0.
$$

This example shows that these two metric controls alone
do not provide a uniform surjectivity criterion
for collapsing families. It does not establish the necessity of the
geometric assumptions in Theorem~\ref{thm:B}, since the target
space also varies here. In the proof of the theorem,
surjectivity follows from homotopy control
of the approximation and top-dimensional homology.
\end{remark}
\begin{remark}[Geometric bounds do not suffice]
The factor persistence conclusion in
Theorem~\ref{thm:B} may fail if $GH$-stability is omitted,
even when all the dynamical systems are defined on the same
manifold. Indeed, let
$f=\operatorname{id}_{\mathbb S^1}$ and let $R_{\theta_j}$ be irrational
rotations with $\theta_j\to0$. Then
$$
d_{GH^0}(f,R_{\theta_j})\longrightarrow0.
$$
If a continuous map $h:\mathbb S^1\to\mathbb S^1$
satisfied
$$
h\circ R_{\theta_j}=h,
$$
it would be constant on each orbit. Since the orbits of
$R_{\theta_j}$ are dense, $h$ would be constant and could not
be surjective. Thus, geometric restrictions
alone do not guarantee the existence of the factor.
\end{remark}

The following example shows how the two theorems apply
to the same system, depending on whether arbitrary perturbations
are allowed or perturbations are restricted to a fixed geometric class.

\begin{example}\label{ex:gato}
This example illustrates how Theorems~\ref{thm:A} and~\ref{thm:B}
lead to different entropy behavior depending on
the restrictions imposed on the phase spaces of the perturbations. Consider the automorphism $f_A:\mathbb T^2\to\mathbb T^2$ of the torus
induced by
$$
A=
\begin{pmatrix}
2&1\\
1&1
\end{pmatrix}.
$$
Its eigenvalues are $\lambda_\pm=\dfrac{3\pm\sqrt5}{2}$, and $h_{\rm top}(f_A)=\log\lambda_+>0$. Since $f_A$ is Anosov, it is expansive and has the
shadowing property \cite{aokiHiraide1994}; hence, it is $GH$-stable
\cite[Theorem~4]{arbieto2017}. Theorem~\ref{thm:A} provides homeomorphisms
$g_j:Y_j\to Y_j$ with zero entropy such that
$$
d_{GH^0}(f_A,g_j)\longrightarrow0.
$$
By $GH$-stability, passing to a subsequence if necessary,
there exist $\varepsilon_j\to0$ and continuous $\varepsilon_j$-isometries
$h_j:Y_j\to\mathbb T^2$ satisfying $f_A\circ h_j=h_j\circ g_j$.
None of these maps can be surjective, since
monotonicity of entropy under factors would imply
$$
0<h_{\rm top}(f_A)\leq h_{\rm top}(g_j)=0.
$$
Thus, even for this hyperbolic automorphism, the maps
provided by $GH$-stability may have arbitrarily small
distortion and arbitrarily dense image without
being surjective. On the other hand, fix $K,i_0,D>0$ so that the flat torus
belongs to $\mathcal R(2,K,i_0,D)$.
Theorem~\ref{thm:B} guarantees that every homeomorphism
$g:Y\to Y$ sufficiently close to $f_A$ in the
$GH^0$ distance, with $Y\in\mathcal R(2,K,i_0,D)$, has $f_A$
as a factor and satisfies
$$
h_{\rm top}(g)\geq
\log\left(\frac{3+\sqrt5}{2}\right).
$$
Consequently, for any sequence of homeomorphisms
$g_j:Y_j\to Y_j$ with zero entropy converging to $f_A$
in the $GH^0$ distance, we have $Y_j\notin\mathcal R(2,K,i_0,D)$
for all sufficiently large $j$. The zero-entropy
approximation in Theorem~\ref{thm:A} is therefore
incompatible with remaining in a fixed geometric class
of this type.
\end{example}

\bibliographystyle{amsplain}
\bibliography{References}
\end{document}